\documentclass[leqno,a4paper,10pt]{amsart}
\usepackage{amsmath,amsthm,amssymb,amsfonts,enumerate,color}

\newtheorem{theorem}{Theorem}[section]

\newtheorem{lemma}[theorem]{Lemma}
\newtheorem{definition}[theorem]{Definition}

\newcommand{\norm}[1]{\left\Vert#1\right\Vert}
\newcommand{\abs}[1]{\left\vert#1\right\vert}

\def \RN   {\mathbb{R}^n}

\allowdisplaybreaks
\title[Variation operators in high order Schr\"odinger setting]{Boundedness of variation operators associated with the heat semigroup generated by high order Schr\"odinger type operators}


\thanks{$^*${Corresponding author.}
The first author was supported by the National Natural Science Foundation of China (No. 11701453). The second author was supported by the Natural Science Foundation of Zhejiang Province(Grant No. LY18A010006), the first Class Discipline of Zhejiang - A (Zhejiang Gongshang University- Statistics) and the State Scholarship Fund(No. 201808330097).
}

\address{Suying Liu \endgraf Department of Applied Mathematics, Northwest Polytechnical University,
Xi'an 710072, People's Republic of China \endgraf \emph{E-mail address: suyingliu@nwpu.edu.cn}}
\address{Chao Zhang \endgraf Department of Mathematics,
Zhejiang Gongshang University, Hangzhou, 310018, People's Republic of China \endgraf \emph{E-mail address: zaoyangzhangchao@163.com}}

\author[S. Liu and C. Zhang]{Suying Liu and Chao Zhang*}

\begin{document}

\vskip 3cm
\begin{abstract}
In this paper, we derive the $L^p$-boundedness of the variation operators associated with the heat semigroup which is generated by the high order Schr\"odinger type operator $(-\Delta)^2+V^2$. Further more, we prove the boundedness of the variation operators on Morrey spaces. In the proof of the main results, we always make use of the variation inequalities associated with the heat semigroup generated by the biharmonic operator $(-\Delta)^2.$
\end{abstract}

\maketitle

 \vskip3mm
\section{\bf Introduction}

Variation inequalities have been the subject of many recent research papers in probability, ergodic theory and harmonic analysis.
The first variation inequality  was proved  by L\'epingle \cite{Lepingle} in martingale theory. Bourgain \cite{Bourgain} proved the variation inequality for the ergodic averages of a dynamic system. Bourgain's work has inaugurated a new research direction in ergodic theory and harmonic analysis. And then, Campbell, Jones, Reinhold and Wierdl \cite{CJRW2000} proved the variation inequalities for the Hilbert transform. Since then many other publications came to enrich the literature on this subject in Harmonic Analysis (see \cite{Bui, CJRW, GillTorrea, JR, JSW, JWang, MTX} and so on).

 Let $\{T_t\}_{t>0}$ be a family of operator such that the limit $\displaystyle \lim_{t\rightarrow 0}=Tf(x)$ exists in some sense.
A classical method of measuring the speed of convergence of the family $\{T_t\}_{t>0}$ is to consider the ``square function" of the type $\displaystyle \Big(\sum_{i=1}^\infty|T_{t_i}f-T_{t_{i+1}}f|^2\Big)^{1/2}$, where $t_{i}\searrow 0$, or the more generally variation operator
$\mathcal{V}_\rho(T_t)$, where $\rho >2$, is given by
$$\mathcal{V}_\rho(T_t)(f)(x):=\sup_{t_{i}\searrow 0}\Big(\sum^\infty_{i=1}|T_{t_i}f(x)-T_{t_{i+1}}f(x)|^\rho\Big)^{1/\rho},$$
where the supremum is taken over all the positive decreasing sequences $\{t_j\}_{j\in \mathbb N}$ which converge to $0$.
We denote $E_\rho$  the space including all the functions $w: (0, \infty)\rightarrow \mathbb{R}$, such that
\begin{equation*}
\|w\|_{E_\rho}:=\sup_{t_{i}\searrow 0}\Big(\sum^\infty_{i=0}|w(t_i)-w(t_{i+1})|^\rho\Big)^{1/\rho}<\infty,
\end{equation*}
$\|w\|_{E_\rho}$ is a seminorm on $E_\rho$. It can be written as
\begin{equation*}
\mathcal{V}_\rho(T_t)(f)=\|T_tf\|_{E_\rho}.
\end{equation*}

In this paper, we mainly focus on the variation operators associated with the high order Schr\"{o}dinger type operators $\mathcal{L}=(-\Delta)^2+V^2$ in $\RN$ with $n\geq 5$, where the nonnegative potential $V$
belongs to the reverse H\"{o}lder class $RH_q$ for some $q>n/2$, that is, there exists $C>0$, such that
\begin{equation*}
\left(\frac{1}{|B|}\int_B V(x)^qdx\right)^{\frac{1}{q}}\leq \frac{C}{|B|}\int_B V(x)dx,
\end{equation*}
for every ball $B$ in $\RN$.
Some results related with $(-\Delta)^2+V^2$ were firstly considered by Zhong in \cite{Zhong}.  In \cite{Sugano}, Sugano proved the estimation of the fundamental solution, and the  $L^p$-boundedness of some operators related with this  operator. For more results related with this operator, see \cite{CLY, LiuDo, LiuZhang}.

The heat semigroup $e^{-t\mathcal{L}}$ generated by the operator $-\mathcal{L}$ can be written as
\begin{equation*}
e^{-t\mathcal{L}}f(x)=\int_{\RN}\mathcal{B}_t(x,y)f(y)dy, ~~\mbox{for}~~ f\in L^2(\RN), t>0.
\end{equation*}
The kernel of the heat semigroup $e^{-t\mathcal{L}}$
 satisfies the estimate
\begin{equation}\label{B}
|\mathcal{B}_t(x,y)|\leq Ct^{-\frac{n}{4}}e^{-A\frac{|x-y|^{4/3}}{t^{1/3}}},
\end{equation}
for more details see \cite{BD}.

We recall the definition of the function $\gamma(x)$, which plays important roles in the theory of operators associated with $\mathcal{L}$:
\begin{equation}\label{gamma}
\gamma(x)=\sup \Big\{r>0:\frac{1}{r^{n-2}}\int_{B(x,r)} V(x)dx\leq 1\Big\},~~x\in\RN,
\end{equation}
which was introduced by Shen \cite{S}.

For Schr\"{o}dinger operator $L=-\Delta+V$, Betancor et al. established the $L^p$-boundedness properties of the variation operators related with the heat semigroup $\{e^{-tL}\}_{t>0}$ in \cite{BFHR}.
It is a natural and interesting question that whether we can establish the boundedness properties of the variation operators associated with $\{e^{-t\mathcal L}\}_{t>0}$
on $L^p(\RN)$. Our main result is as follows.
\begin{theorem}\label{thonp}
Assume that $V\in RH_{q_0}(\RN)$, where $q_0\in (n/2, \infty)$ and $n\geq 5$. For $\rho>2$, there exists a
 constant $C>$ such that

$$\|\mathcal{V}_{\rho}(e^{-t\mathcal{L}})(f)\|_{L^p(\RN)}\leq C\|f\|_{L^p(\RN)},~~~~~\quad  1<p<\infty.$$

\end{theorem}

We should note that, our results are not contained in the paper of Bui \cite{Bui}, because the estimates of the heat kernel are not the same.

On the other hand, Zhang and Wu \cite{ZW} studied the boundedness of variation operators associated with the heat semigroup $\{e^{-tL}\}_{t>0}$ on Morrey spaces
related to the non-negative potential $V$.  Tang and Dong \cite{TD} introduced Morrey spaces related to non-negative potential $V$ for extending the
boundedness of Schr\"{o}dinger type operators in Lebesgue spaces.

\begin{definition}\label{Morrey}
Let $1\leq p<\infty, \alpha\in \mathbb{R}$ and $0\leq \lambda<n$. For $f\in L^p_{loc}(\RN)$ and $V\in RH_q (q>1)$,
we say $f\in L^{p,\lambda}_{\alpha, V}(\RN)$, if
 \begin{equation*}
\|f\|^p_{L^{p,\lambda}_{\alpha, V}}=\sup_{B(x_0,r)\subset\RN}\Big(1+\frac{r}{\gamma(x_0)}\Big)^\alpha r^{-\lambda}\int_{B(x_0,r)}|f(x)|^pdx<\infty,
\end{equation*}
where $B(x_0,x)$ denotes a ball centered at $x_0$ and with radius $r$, $\gamma(x_0)$ is defined as in (\ref{gamma}).
\end{definition}

For more information about the Morrey spaces associated with differential operators, see \cite{HuangZ, Song, YSY}.

Then, we can also obtain the boundedness of the variation operators associated to the heat semigroup $\{e^{-t\mathcal L}\}_{t>0}$ on Morry spaces.

\begin{theorem}\label{thonm}
Let $V\in RH_{q_0}(\RN)$ for $q_0\in (n/2, \infty)$, $n\geq 5$ and $\rho>2$. Assume that $\alpha\in \mathbb{\mathbb{R}}$ and $\lambda\in (0,n)$.
There exists a constant $C>0$ such that

$$\|\mathcal{V}_{\rho}(e^{-t\mathcal{L}})(f)\|_{L^{p,\lambda}_{\alpha, V}(\RN)}\leq C\|f\|_{L^{p,\lambda}_{\alpha, V}(\RN)},~~~~~~\quad  1< p<\infty.$$
\end{theorem}

The organization of the paper is as follows. Section \ref{Sec:2} is devoted to giving the proof of Theorem \ref{thonp}. In order to prove this theorem, we should study the strong $L^p$-boundednessof the variation operators associated to $\{e^{-t\Delta^2}\}_{t>0}$ first.
 We will give the proof of Theorem \ref{thonm} in Section \ref{Sec:3}.
We also obtain the strong
$L^p(\RN)$ estimates $(p>1)$ of the generalized Poisson operators $\mathcal{P}_{t,\mathcal{L}}^\sigma$ on $L^p$ as well as  Morrey spaces related to non-negative potential $V$, respectively, in Section \ref{Sec:2} and Section \ref{Sec:3}.

\medskip
Throughout this paper, the symbol $C$ in an inequality always denotes a constant which may depend on some indices, but never on the functions $f$ in consideration.

\vskip3mm
\section{ Variation inequalities related to $\{e^{-t\mathcal{L}}\}_{t>0}$ on $L^p$ spaces}\label{Sec:2}

In this section, we first recall some properties of biharmonic heat kernel. With these kernel estimates, we will
give the proof of $L^p$-boundedness properties of the variation operators related to $\{e^{-t\Delta^2}\}_{t>0}$,
which is crucial in the proof of Theorem \ref{thonp}.

\subsection{\bf Biharmonic heat kernel.}\label{sec:kernel}
Consider the following Cauchy problem for the biharmonic heat equation
\begin{equation*}%\label{hk}
\begin{cases}
(\partial_t+\Delta^2)u(x,t)=0 \ \ \  \text{in } \  \mathbb R_+^{n+1}\\
u(x,0)=f(x) \ \ \  \ \ \  \ \ \  \  \  \text{in } \  \mathbb R^{n}.
\end{cases}
\end{equation*}
Its solution is  given by
$$u(x,t)=e^{-t{\Delta^2}}f(x)=\int_{\RN}b(x-y,t)f(y)dy,$$
where $$b(x,t)=t^{-\frac{n}{4}}g(\eta),$$
where $\eta={x}{t^{-\frac{1}{4}}}$ and
\begin{equation}\label{g}
g(\eta)=(2\pi)^{-\frac{n}{2}}\int_{\RN}e^{i\eta k-|k|^4}dk=\alpha_n|\eta|\int_0^\infty e^{-s^4}J_{(n-2)/2}(|\eta|s)ds,~~~~\eta\in\RN,
\end{equation}
where $J_v$ denotes the $v$-th Bessel function and $\alpha_n>0$ is a normalization constant such that
\begin{equation*}
\int_{\RN}g(\eta)d\eta=1.
\end{equation*}
Then, we have the following several results by classical analysis(for details, see \cite{SW}):
 \begin{itemize}
\item[(1)]
 If $f\in L^p(\RN)$, $1\le p\le \infty,$ then $$\displaystyle \lim_{t\rightarrow 0}u(x,t)=f(x)\quad  a.e. \quad  x\in \RN,$$ and
\begin{equation*}
\norm{u(\cdot, t)}_{L^p(\RN)}\le C\norm{f}_{L^p(\RN)}.
\end{equation*}
\item [(2)] If $1\le p<\infty,$ then
\begin{equation*}
\norm{u(\cdot, t)-f}_{L^p(\RN)}\rightarrow 0, \quad  \hbox{when} \quad t\rightarrow 0^+.
\end{equation*}
\end{itemize}

And $g(\eta)$ satisfies the following estimates
 $$|g(\eta)|\leq C\Big(1+|\eta|\Big)^{-\frac{n}{3}}e^{-A_1 |\eta|^{\frac{4}{3}}},$$
\begin{equation}\label{dg}
\frac{d^m g}{d\eta^m}(\eta)\leq C_m\Big(1+|\eta|\Big)^{-\frac{n-m}{3}}e^{-A_1 |\eta|^{\frac{4}{3}}},~~~~~~m\in \mathbb{N},
\end{equation}
see \cite{LK}.

We should note that the heat semigroup $e^{-t\Delta^2}$ doesn't have the positive preserving property, i.e., when $f\ge 0$, then $e^{-t\Delta^2}f\ge 0$ maybe not established. So, the boundedness of the variation operators associated to $\{e^{-t\Delta^2}\}_{t>0}$ cannot be deduced by the results in \cite{JR}.

For the heat kernel $b(x,t)$ of the semigroup $e^{-t\Delta^2}$, we have the following estimates.
\begin{lemma}\label{K}
For every $t>0$ and $\RN$, we have
\begin{equation}\label{k}
\big|b(x,t)\big|\leq Ct^{-\frac{n}{4}}e^{-A_1 |xt^{-\frac{1}{4}}|^\frac{4}{3}},
\end{equation}
\begin{equation}\label{bxt}
\big|\partial_t^l\nabla_x^kb(x,t)\big|\leq C\Big(t^{\frac{1}{4}}+|x|\Big)^{-n-k-4l},\ \  \forall k,l\geq 1,
\end{equation}
\begin{equation}\label{kt}
\big|\frac{\partial }{\partial t}b(x,t)\big|\leq Ct^{-\frac{n}{4}-1}\Big(1+|xt^{-\frac{1}{4}}|\Big)^{-\frac{n}{3}+\frac{4}{3}}e^{-A_1 |xt^{-\frac{1}{4}}|^\frac{4}{3}},
\end{equation}
\begin{equation}\label{kx}
\big|\nabla_x b(x,t)\big|\leq Ct^{-\frac{n+1}{4}}\Big(1+|xt^{-\frac{1}{4}}|\Big)^{-\frac{n-1}{3}}e^{-A_1 |xt^{-\frac{1}{4}}|^\frac{4}{3}},
\end{equation}
where $\displaystyle A_1=\frac{3\cdot2^{1/3}}{16}$.
\end{lemma}
\begin{proof}
 For (\ref{k}) and (\ref{bxt}), see Lemma 2.4 in \cite{LK}. From (\ref{g}) and (\ref{dg}), and through some simple calculations, we can derive (\ref{kt}) and (\ref{kx}).
\end{proof}
\medskip
\subsection{\bf Variation inequalities related to $e^{-t\Delta^2}$}

By (1) in Section \ref{sec:kernel}, we know that the operator $e^{-t\Delta^2}$ is a contraction on  $L^1(\RN)$ and $L^\infty(\RN)$. So, $e^{-t\Delta^2}$ is contractively regular. And then, by \cite[Corollary 3.4]{MXu}, we have the following theorem(For more details, see \cite{MXu}).

\begin{theorem}\label{E}
For $\rho>2$,  there exists a constant $C>0$ such that
\begin{equation*}
\|\mathcal{V}_{\rho}(e^{-t\Delta^2})(f)\|_{L^p(\RN)}\leq C\|f\|_{L^p(\RN)},~~~~~~\forall f\in L^p(\RN),\quad 1<p<\infty.
\end{equation*}
\end{theorem}

\vskip3mm
\subsection{\bf  Variation inequalities related to $\{e^{-t\mathcal{L}}\}_{t>0}$ }

First, we recall some properties of the auxiliary function $\gamma(x)$, which will be used later.
\begin{lemma}[\cite{S}]\label{veq}
Let $V\in RH_{\frac{n}{2}}(\RN)$, then there exist $C$ and $k_0>1$, such that for all $x, y\in \RN$,
\begin{equation*}
\frac{1}{C}\gamma(x)\Big(1+\frac{|x-y|}{\gamma(x)}\Big)^{-k_0}\leq \gamma(y)\leq C\gamma(x)\Big(1+\frac{|x-y|}{\gamma(x)}\Big)^{\frac{k_0}{k_0+1}}.
\end{equation*}
In particular, $\gamma(x)\sim\gamma(y)$ if $|x-y|<C\gamma(x)$.
\end{lemma}

\begin{lemma}[Lemma 2.7 in \cite{CLY}]\label{V}
Let $V\in RH_{q_0}(\RN)$ and $\delta=2-n/q_0$, where $q_0\in (n/2, \infty)$ and $n\geq 5$.
Then there exists a positive constant $C$ such that for all $x, y\in \RN$ and $t\in (0, \gamma^4(x)]$,
\begin{equation*}
\int_{\RN}\frac{V^2(y)}{t^{n/4}}e^{-A_4\frac{|x-y|^{4/3}}{t^{1/3}}}dy\leq Ct^{-1}\Big(\frac{t^{1/4}}{\gamma(x)}\Big)^{2\delta},
\end{equation*}
where $A_4=\min \{A, A_1\}$, and $A, A_1$ are constants, respectively, as in (\ref{B}) and (\ref{k}).
\end{lemma}

And we can prove the following kernel estimates of $e^{-t\mathcal{L}}$.
\begin{lemma}\label{k6}
For every $N\in \mathbb{N}$, there exist positive constants $C$, $A_2$ and $A_3$ such that for all $x,y\in \RN$
and $0<t<\infty$,

\smallskip
\noindent(i)~~$\displaystyle|\mathcal{B}_t(x,y)|\leq Ct^{-\frac{n}{4}}\Big(1+\frac{\sqrt{t}}{\gamma^2(x)}+\frac{\sqrt{t}}{\gamma^2(y)}\Big)^{-N}
e^{-A_2\frac{|x-y|^{4/3}}{t^{1/3}}}$,

\smallskip
\noindent(ii)
$\displaystyle \Big|\frac{\partial}{\partial t}\mathcal{B}_t(x,y)\Big|\leq Ct^{-\frac{n+4}{4}}\Big(1+\frac{\sqrt{t}}{\gamma^2(x)}+\frac{\sqrt{t}}{\gamma^2(y)}\Big)^{-N}
e^{-A_3\frac{|x-y|^{4/3}}{t^{1/3}}},
$\\
where $A_2=A_1/2$, and $A_3<A_2$.
\end{lemma}

\begin{proof}
For $(i)$, see Theorem 2.5 of \cite{CLY}.

Now we give the proof of $(ii)$. As $\mathcal{L}=(-\Delta)^2+V^2$ is a nonnegative self-adjoint operator, we can extend the semigroup $\displaystyle \{e^{-t\mathcal{L}}\}$ to a holomorphic semigroup $\{T_\xi\}_{\xi\in \Delta_{\pi/4}}$ uniquely. The kernel $\mathcal{B}_\xi (x,y)$ of $T_\xi$ satisfies
\begin{equation}\label{eq:kerHolo}
|\mathcal{B}_\xi(x,y)|\le C_N (R\xi)^{-n/4}\Big(1+\frac{\sqrt{R\xi}}{\gamma^2(x)}+\frac{\sqrt{R\xi}}{\gamma^2(y)}\Big)^{-N}e^{-C\frac{|x-y|^{4/3}}{(R\xi)^{1/3}}}.
\end{equation}
The Cauchy integral formula combined with \eqref{eq:kerHolo} gives
\begin{equation*}
\Big|\frac{\partial}{\partial t} \mathcal{B}_t(x,y)\Big|=\Big|\frac{1}{2\pi}\int_{\abs{\xi-t}=t/10}\frac{\mathcal{B}_\xi(x,y)}{{(\xi-t)}^2}d\xi\Big|\le \frac{C_N}{t^{n/4+1}}\Big(1+\frac{\sqrt t}{\gamma^2(x)}+\frac{\sqrt t}{\gamma^2(y)}\Big)^{-N}e^{-C\frac{{x-y}^{4/3}}{t^{1/3}}}.
\end{equation*}
Then, we complete the proof.
\end{proof}

With the estimates above, we can give the proof of Theorem \ref{thonp}.

\begin{proof}[ Proof of Theorem \ref{thonp}]
For $f\in L^p(\RN), 1\leq p<\infty$, we consider the following local operators
\begin{equation*}
e_{loc}^{-t\mathcal{L}}f(x)=\int_{|x-y|<\gamma(x)}\mathcal{B}_t(x,y)f(y)dy,~~~~~~x\in \RN,
\end{equation*}
and
\begin{equation*}
e_{loc}^{-t\Delta^2}f(x)=\int_{|x-y|<\gamma(x)}b_t(x-y)f(y)dy,~~~~~~x\in \RN.
\end{equation*}
Then, we have
\begin{equation*}
\mathcal{V}_\rho(e^{-t\mathcal{L}})(f)
\leq \mathcal{V}_\rho(e_{loc}^{-t\mathcal{L}}-e_{loc}^{-t\Delta^2})(f)+\mathcal{V}_\rho(e_{loc}^{-t\Delta^2})(f)
+\mathcal{V}_\rho(e^{-t\mathcal{L}}-e_{loc}^{-t\mathcal{L}})(f)\\
=: J_1+J_2+J_3.
\end{equation*}
Let us analyze term $J_2$ first.
\begin{align*}
J_2&\leq  \Big(\sum_{j=0}^\infty \Big|e^{-t_{j}\Delta^2}(f)(x)- e^{-t_{j+1}\Delta^2}(f)(x)\Big|^\rho\Big)^{\frac{1}{\rho}}\\
&\qquad+\Big(\sum_{j=0}^\infty \Big|\int_{|x-y|>\gamma(x)}b(x-y,t_j)-b(x-y,t_{j+1})dy\Big|^\rho\Big)^{\frac{1}{\rho}}\\
&\leq  \mathcal{V}_\rho(e^{-t\Delta^2})(f)(x)+\sup_{\varepsilon>0}\Big\|\int_{|x-y|>\varepsilon}b(x-y,t)f(y)\Big\|_{E_\rho}.
\end{align*}
We consider the operator defined by
\begin{align*}
T: L^2(\RN)&\rightarrow L^2_{E_\rho}(\RN)\\
f&\rightarrow  Tf(x)=\int_{\RN}b(x-y,t)f(y)dy,
\end{align*}
which is bounded from $L^2(\RN)$ into $L^2_{E_\rho}(\RN)$ according to Theorem \ref{E}.
Moreover, $T$ is a Calder\'{o}n-Zygmund operator with the $E_\rho$-valued kernel $b(x-y,t)$. In fact, the kernel $b(x-y,t)$ has the following two properties:
\begin{enumerate}[(1)]
\item By (\ref{kt}), we have
\begin{align*}
\quad \|b(x-y,\cdot)\|_{E_\rho}&\leq \sup_{t_j\searrow 0}\sum_{j=0}^\infty\int^{t_j}_{t_{j+1}}\left|\frac{\partial}{\partial t}b(x-y,t)\right|dt\\
 &\leq C\int_0^\infty\Big|\frac{\partial}{\partial t}b(x-y,t)\Big|dt\\
&\leq C\int_0^{|x-y|^4}\Big|\frac{\partial}{\partial t}b(x-y,t)\Big|dt+ C\int_{|x-y|^4}^\infty\Big|\frac{\partial}{\partial t}b(x-y,t)\Big|dt\\
&\leq C\int_0^{|x-y|^4}t^{-\frac{n}{4}-1}\left(\frac{t^{1/3}}{|x-y|^{4/3}}\right)^{\frac{3}{4}(n+4)}dt+\int_{|x-y|^4}^\infty t^{-\frac{n}{4}-1}dt\\
&\leq C|x-y|^{-n},\quad x,y\in\RN, t>0.
\end{align*}
\item Proceeding a similar way together with (\ref{bxt}), we have
\begin{align*}
\quad \Big\|\frac{\partial}{\partial x}b(x-y,\cdot)\Big\|_{E_\rho}+\Big\|\frac{\partial}{\partial y}b(x-y,\cdot)\Big\|_{E_\rho}\leq C|x-y|^{-n-1}, \quad x,y\in\RN, t>0.
\end{align*}
\end{enumerate}
Thus,  by proceeding as in the proof of  \cite[Proposition 2, p. 34 and Corollary 2, p. 36,]{Stein},  we can prove that the maximal operator $T^*$ defined by
$$T^*=\sup_{\varepsilon>0}\Big\|\int_{|x-y|>\varepsilon}b(x-y,t)f(y)\Big\|_{E_\rho}$$
is bounded on $L^p(\RN)$  for every $1<p<\infty$. Combining Theorem \ref{E}, we conclude that $\mathcal{V}_\rho(e_{loc}^{-t\Delta^2})$
 is bounded from $L^p(\RN)$ into itself for every $1<p<\infty$.

Next, we consider term $J_3$.
\begin{align*}
J_3&=\sup_{t_j\searrow 0}\left(\sum_{j=0}^\infty \Big|\int_{|x-y|>\gamma(x)}(\mathcal{B}_{t_j}(x,y)-\mathcal{B}_{t_{j+1}}(x,y))f(y)dy\Big|^\rho\right)^{\frac{1}{\rho}}\\
&\leq \sup_{t_j\searrow 0}\sum_{j=0}^\infty\int_{|x-y|>\gamma(x)}|f(y)|\int^{t_j}_{t_{j+1}}\left|\frac{\partial}{\partial t}\mathcal{B}_t(x,y)\right|dtdy\\
&\leq \int_{|x-y|>\gamma(x)}|f(y)|\int_0^\infty\Big|\frac{\partial}{\partial t}\mathcal{B}_t(x,y)\Big|dtdy\\
&\leq \int_{|x-y|>\gamma(x)}|f(y)|\left(\int_{\gamma^4(x)}^\infty\Big|\frac{\partial}{\partial
t}\mathcal{B}_t(x,y)\Big|dt+\int_0^{\gamma^4(4)}\Big|\frac{\partial}{\partial t}\mathcal{B}_t(x,y)\Big| dt\right)dy\\
&:= J_{31}+J_{32}.
\end{align*}
To estimate $J_{31}$, by Lemma \ref{k6} with $N=n+2$ and changing variables, we have
\begin{align*}
J_{31}&\leq C\int_{|x-y|>\gamma(x)}|f(y)|\int_{\gamma^4(x)}^\infty t^{-\frac{n+4}{4}}\left(1+\frac{\sqrt{t}}{\gamma^2(x)}\right)^{-n-2}
e^{-A_3\frac{|x-y|^{4/3}}{t^{1/3}}}dtdy\\
&\leq C\int_{|x-y|>\gamma(x)}|f(y)|\int_1^\infty \frac{1}{\gamma^n(x)}\frac{1}{u^{n/2+1}}\frac{1}{(1+u)^{n+2}}e^{-A_3\frac{|x-y|^{4/3}}{u^{2/3}\gamma^{4/3}(x)}}dudy\\
&\leq C\frac{1}{\gamma^n(x)}\int_{|x-y|>\gamma(x)}|f(y)|\int_1^\infty \frac{1}{u^{n/2+1}}\frac{1}{(1+u)^{n+2}}\left(\frac{u^{2/3}\gamma^{4/3}(x)}{|x-y|^{4/3}}\right)^{\frac{3(n+2)}{4}}dudy\\
&\leq C\frac{1}{\gamma^n(x)}\int_{|x-y|>\gamma(x)}|f(y)|\left(\frac{\gamma(x)}{|x-y|}\right)^{n+2}dy\\
&\leq C\frac{1}{\gamma^n(x)}\sum_{k=0}^\infty\int_{2^k\gamma(x)<|x-y|\leq 2^{k+1}\gamma(x)}|f(y)|\left(\frac{\gamma(x)}{|x-y|}\right)^{n+2}dy\\
&\leq C\sum_{k=0}^\infty\frac{1}{2^{2k}(2^k\gamma(x))^n}\int_{|x-y|\leq 2^{k+1}\gamma(x)}|f(y)|dy\leq CM(f)(x),
\end{align*}
where $M(f)$ is the Hardy-Littlewood maximal function of $f$.
For $J_{32}$, by Lemma \ref{k6} we have
\begin{align*}
J_{32}&\leq C\int_{|x-y|>\gamma(x)}|f(y)|\int_0^{\gamma^4(x)} t^{-\frac{n+4}{4}}e^{-A_3\frac{|x-y|^{4/3}}{t^{1/3}}}dtdy\\
&\leq C\int_0^{\gamma^4(x)}\frac{e^{-c\frac{\gamma^{4/3}(x)}{t^{1/3}}}}{t}\int_{\RN}t^{-\frac{n}{4}}e^{-C\frac{|x-y|^{4/3}}{t^{1/3}}}|f(y)|dy dt\\
&\leq C\sup_{t>0}t^{-\frac{n}{4}}\int_{\RN}e^{-C\frac{|x-y|^{4/3}}{t^{1/3}}}|f(y)|dtdy\leq CM(f)(x).
\end{align*}
Thus from the estimates $J_{31}$ and $J_{32}$, we have $J_3\leq C M(f)(x)$,
 which implies that the operator $\mathcal{V}_\rho(e^{-t\mathcal{L}}-e_{loc}^{-t\mathcal{L}})(f)$ is bounded
 from $L^p(\RN)$ into itself for every $1<p<\infty$.

Finally, we consider the term $J_1$.
\begin{align*}
&J_1\\&=\sup_{t_j\searrow 0}\left(\sum_{j=0}^\infty \Big|\int_{|x-y|<\gamma(x)}\left((\mathcal{B}_{t_j}(x,y)-b(x-y,t_j))-(\mathcal{B}_{t_{j+1}}(x,y)-b(x-y,t_{j+1}))\right)f(y)dy\Big|^\rho\right)^{\frac{1}{\rho}}\\
&\leq \sup_{t_j\searrow 0}\sum_{j=0}^\infty\int_{|x-y|<\gamma(x)}|f(y)|\int^{t_j}_{t_{j+1}}\Big|\frac{\partial}{\partial t}\big(\mathcal{B}_t(x,y)-b(x-y,t)\big)\Big|dtdy\\
&\leq \int_{|x-y|<\gamma(x)}|f(y)|\int^{\infty}_{0}\Big|\frac{\partial}{\partial t}\big(\mathcal{B}_t(x,y)-b(x-y,t)\big)\Big|dtdy\\
&\leq \int_{|x-y|<\gamma(x)}|f(y)|\int^{\infty}_{\gamma^4(x)}\Big|\frac{\partial}{\partial t}\big(\mathcal{B}_t(x,y)-b(x-y,t)\big)\Big|dtdy\\
&\quad +\int_{|x-y|<\gamma(x)}|f(y)|\int_0^{\gamma^4(x)}\Big|\frac{\partial}{\partial t}\big(\mathcal{B}_t(x,y)-b(x-y,t)\big)\Big|dtdy\\
&=:J_{11}+J_{12},\quad \quad ~x\in \RN.
\end{align*}
Applying Lemma \ref{K} and Lemma \ref{k6}, we have
\begin{align*}
J_{11}
&\leq C\int_{|x-y|<\gamma(x)}|f(y)|\int^{\infty}_{\gamma^4(x)} t^{-\frac{n+4}{4}}e^{-C\frac{|x-y|^{4/3}}{t^{1/3}}}dtdy\\
&\leq C\int_{|x-y|<\gamma(x)}|f(y)|\int^{\infty}_{\gamma^4(x)} t^{-\frac{n+4}{4}}dtdy\\
&\leq C\frac{1}{\gamma(x)^n}\int_{|x-y|<\gamma(x)}|f(y)|dy\leq CM(f)(x).
\end{align*}
The  formula (2.7) in \cite{CLY}  implies
\begin{align*}
&\frac{d}{dt}\Big(e^{-t\mathcal{L}}-e^{-t\Delta^2}\Big)\\
&=-e^{-\frac{t}{2}\Delta^2}V^2e^{-\frac{t}{2}\mathcal{L}}
-\int_0^{{\frac{t}{2}}}\frac{d}{dt}e^{-(t-s)\Delta^2}V^2e^{-s\mathcal{L}}ds-\int_{{\frac{t}{2}}}^{t}e^{-(t-s)\Delta^2}V^2\frac{d}{ds}e^{-s\mathcal{L}}ds.
\end{align*}
Then we have
\begin{align*}
&\frac{\partial}{\partial t}\big(\mathcal{B}_t(x,y)-b(x-y,t)\big)\\
&=-\int_{\RN}V^2(z)b(x-z,{t}/{2})\mathcal{B}_{\frac{t}{2}}(z,y)dz-\int_0^{\frac{t}{2}}\int_{\RN}V^2(z)\frac{\partial}{\partial t}b(x-z,t-s)\mathcal{B}_s(z,y)dzds\\
&\quad -\int^t_{\frac{t}{2}}\int_{\RN}V^2(z)b(x-z,t-s)\frac{\partial}{\partial s}\mathcal{B}_s(z,y)dzds\\
&=K_1(x,y,t)+K_2(x,y,t)+K_3(x,y,t),\qquad x,y\in \RN \text{and}~~ t>0.
\end{align*}
We rewrite $J_{12}$ as
\begin{align*}
J_{12}=\sum_{k=1}^3\int_{|x-y|<\gamma(x)}|f(y)|\int_0^{\gamma^4(x)}|K_m(x,y,t)|dtdy=\sum_{k=1}^3T_mf(x).
\end{align*}
Using (\ref{k}), Lemmas \ref{k6} and \ref{V}, we obtain
\begin{align*}
\int_0^{\gamma^4(x)}|K_1(x,y,t)|dt
&\leq C \int_0^{\gamma^4(x)}\int_{\RN}V^2(z)t^{-\frac{n}{4}}e^{-A_1\frac{|x-z|^{4/3}}{t^{1/3}}}t^{-\frac{n}{4}}e^{-A_2\frac{|y-z|^{4/3}}{t^{1/3}}}dzdt\\
&\leq C \int_0^{\gamma^4(x)}t^{-\frac{n}{4}}e^{-A_2\frac{|x-y|^{4/3}}{t^{1/3}}}\int_{\RN}t^{-\frac{n}{4}}e^{-A_2\frac{|x-z|^{4/3}}{t^{1/3}}}V^2(z)dzdt\\
&\leq C \int_0^{\gamma^4(x)}t^{-\frac{n}{4}-1}e^{-A_2\frac{|x-y|^{4/3}}{t^{1/3}}}\left(\frac{t^{1/4}}{\gamma(x)}\right)^{2\delta}dt.
\end{align*}
As a consequence,
\begin{align*}
|T_1(f)(x)|&\leq  C\int_{|x-y|<\gamma(x)}|f(y)|\int_0^{\gamma^4(x)}t^{-\frac{n}{4}-1}e^{-A_2\frac{|x-y|^{4/3}}{t^{1/3}}}\left(\frac{t^{1/4}}{\gamma(x)}\right)^{2\delta}dtdy\\
&\leq  C\int_0^{\gamma^4(x)}\frac{t^{-1+\delta/2}}{\gamma(x)^{2\delta}}\frac{1}{t^{\frac{n}{4}}}\int_{\RN}|f(y)|e^{-A_2\frac{|x-y|^{4/3}}{t^{1/3}}}dy dt\\
&\leq C\sup_{t>0}\ \frac{1}{t^{\frac{n}{4}}}\int_{\RN}|f(y)|e^{-A_2\frac{|x-y|^{4/3}}{t^{1/3}}}dy\leq CM(f)(x).
\end{align*}
Next, we note that, when $0<s<t/2$,  $t-s\sim t$. And by (\ref{kt}), Lemmas \ref{k6} and \ref{V}, we have
\begin{align*}
&\int_0^{\gamma^4(x)}|K_2(x,y,t)|dt\\
&\leq C \int_0^{\gamma^4(x)}\int_0^{t/2}\int_{\RN}V^2(z)\frac{1}{(t-s)^{\frac{n}{4}+1}}e^{-A_1\frac{|x-z|^{4/3}}{(t-s)^{1/3}}}\frac{1}
{s^{\frac{n}{4}}}e^{-A_2\frac{|y-z|^{4/3}}{s^{1/3}}}dzdsdt\\
&\leq C \int_0^{\gamma^4(x)}\frac{1}{t^{\frac{n}{4}+1}}\int_0^{t/2}\int_{\RN}V^2(z)
e^{-\frac{A_1}{4}\frac{|x-z|^{4/3}}{t^{1/3}}}\frac{1}{s^{\frac{n}{4}}}e^{-\frac{A_1}{2}\frac{|y-z|^{4/3}}{s^{1/3}}}dzdsdt\\
&\leq C \int_0^{\gamma^4(x)}\frac{1}{t^{\frac{n}{4}+1}}\int_0^{t/2}\int_{\RN}e^{-\frac{A_1}{4}\frac{|x-z|^{4/3}+|y-z|^{4/3}}{t^{1/3}}}
\frac{1}{s^{\frac{n}{4}}}V^2(z)e^{-\frac{A_1}{4}\frac{|y-z|^{4/3}}{s^{1/3}}}dzdsdt\\
&\leq C \int_0^{\gamma^4(x)}\frac{1}{t^{\frac{n}{4}+1}}\int_0^{\frac{4^{1/3}t}{2}}e^{-\frac{A_1}{4}\frac{|x-y|^{4/3}}{t^{1/3}}}\int_{\RN}
\frac{1}{s^{\frac{n}{4}}}V^2(z)e^{-A_1\frac{|y-z|^{4/3}}{s^{1/3}}}dzdsdt\\
&\leq C \int_0^{\gamma^4(x)}\frac{1}{t^{\frac{n}{4}+1}}e^{-\frac{A_1}{4}\frac{|x-y|^{4/3}}{t^{1/3}}}\int_0^{\frac{4^{1/3}t}{2}}
\frac{s^{-1+\delta/2}}{\gamma(y)^{2\delta}}dsdt\\
&\leq C \frac{1}{\gamma(x)^{2\delta}}\int_0^{\gamma^4(x)}\frac{1}{t^{\frac{n}{4}+1-\frac{\delta}{2}}}e^{-\frac{A_1}{4}\frac{|x-y|^{4/3}}{t^{1/3}}}dt.
\end{align*}
Hence,
\begin{align*}
|T_2(f)(x)|&\leq C\frac{1}{\gamma(x)^{2\delta}}\int_{|x-y|<\gamma(x)}|f(y)|\int_0^{\gamma^4(x)}
\frac{1}{t^{\frac{n}{4}+1-\frac{\delta}{2}}}e^{-\frac{A_1}{4}\frac{|x-y|^{4/3}}{t^{1/3}}}dtdy\\
&\leq C\frac{1}{\gamma(x)^{2\delta}}\int_0^{\gamma^4(x)}\frac{1}{t^{1-\frac{\delta}{2}}}\int_{\RN}
\frac{1}{t^{\frac{n}{4}}}e^{-\frac{A_1}{4}\frac{|x-y|^{4/3}}{t^{1/3}}}dtdy\\
&\leq C\sup_{t>0}\frac{1}{t^{\frac{n}{4}}}\int_{\RN}|f(y)|e^{-\frac{A_1}{4}\frac{|x-y|^{4/3}}{t^{1/3}}}dy\leq CM(f)(x).
\end{align*}
As in the previous proof, proceeding a similar computation, we can also obtain
\begin{align*}
|T_3(f)(x)|\leq CM(f)(x).
\end{align*}
Owing to above estimates, we know $J_{12}\leq CM(f)(x)$. Consequently, we have $J_1\leq CM(f)(x)$.
And since $M(f)$ is bounded from $L^p(\RN)$ into itself for every $1<p<\infty$. Then the proof of Theorem \ref{thonp} is complete.
\end{proof}

\vskip3mm
\subsection{ The generalized Poisson operators  $\mathcal{P}_{t,\mathcal{L}}^\sigma$}

For $0<\sigma<1$, the generalized Poisson operators $\mathcal{P}_t^\sigma$ associated to $\mathcal{L}$ is defined as
\begin{align*}
\mathcal{P}_{t,\mathcal{L}}^\sigma f(x)&=\frac{t^{2\sigma}}{4^\sigma \Gamma(\sigma)}\int_0^\infty e^{-\frac{t^2}{4r}}e^{-t\mathcal{L}}f(x)\frac{dr}{r^{1+\sigma}}=\frac{1}{\Gamma(\sigma)}\int_0^\infty e^{-r}e^{-\frac{t^2\mathcal{L}}{4r}}f(x)\frac{dr}{r^{1-\sigma}}.
\end{align*}
We should note that, when $\sigma=1/2,$  $\mathcal{P}_t^\sigma=\mathcal{P}_t^{1/2}$ is just the Poisson semigroup.

For the variation operator associated with  the generalized Poisson operators $\{\mathcal{P}_{t,\mathcal{L}}^\sigma\}_{t>0}$, we have the following theorem.
\begin{theorem}\label{general}
Assume that $V\in RH_{q_0}(\RN)$, where $q_0\in (n/2, \infty)$ and $n\geq 5$. For $\rho>2$, there exists a
 constant $C>0$   such that

$$\|\mathcal{V}_{\rho}(\mathcal{P}_{t,\mathcal{L}}^\sigma)(f)\|_{L^p(\RN)}\leq C\|f\|_{L^p(\RN)},~~~~~~ \quad  1<p<\infty.$$

\end{theorem}
\begin{proof}
 We note that
\begin{multline*}
\mathcal{V}_\rho(\mathcal{P}_{t,\mathcal{L}}^\sigma) f(x)=\|\mathcal{P}_{t,\mathcal{L}}^\sigma f\|_{E_\rho}
=\frac{1}{\Gamma(\sigma)}\Big\|\int_0^\infty e^{-r}e^{-\frac{t^2\mathcal{L}}{4r}}f(x)\frac{dr}{r^{1-\sigma}}\Big\|_{E_\rho}\\
\leq \frac{1}{\Gamma(\sigma)}\int_0^\infty e^{-r}\big\|e^{-\frac{t^2\mathcal{L}}{4r}}f(x)\big\|_{E_\rho}\frac{dr}{r^{1-\sigma}}.
\end{multline*}
Then,  for $1<p<\infty$, by Theorem \ref{thonp} we have
\begin{multline*}
\|\mathcal{V}_\rho(\mathcal{P}_{t,\mathcal{L}}^\sigma) f\|_{L^p(\RN)}
\leq \frac{1}{\Gamma(\sigma)}\int_0^\infty e^{-r}\Big\|\big\|e^{-\frac{t^2\mathcal{L}}{4r}}f(x)\big\|_{E_\rho}\Big\|_{L^p(\RN)}\frac{dr}{r^{1-\sigma}}\\
\leq \frac{C}{\Gamma(\sigma)}\int_0^\infty e^{-r}\|f\|_{L^p(\RN)}\frac{dr}{r^{1-\sigma}}\leq C\|f\|_{L^p(\RN)}.
\end{multline*}
\end{proof}

\vskip3mm
\section{\bf  Variation inequalities in Morrey spaces}\label{Sec:3}
In this section, we will give the proof of Theorem \ref{thonm}.
For convenience, we first recall the the definition of classical Morrey spaces $L^{p,\lambda}(\RN)$, which were introduced by Morrey \cite{M} in 1938.
\begin{definition}
Let $1\leq p<\infty$, $0\leq \lambda<n$. For $f\in L^p_{loc}(\RN)$, we say $f\in L^{p,\lambda}(\RN)$ provided that
\begin{equation*}
\|f\|^p_{L^{p,\lambda}(\RN)}=\sup_{B(x_0,r)\subset\RN}r^{-\lambda}\int_{B(x_0,r)}|f(x)|^pdx<\infty,
\end{equation*}
where $B(x_0,r)$ denotes a ball centered at $x_0$ and with radius $r$.
\end{definition}
In fact, when $\alpha=0$ or $V=0$ and $0<\lambda<n$, the spaces $L^{p,\lambda}_{\alpha,V}(\RN)$ which was defined in Definition \ref{Morrey} are the  classical Morrey spaces $L^{p,\lambda}(\RN)$.

We first establish the $L^{p,\lambda}(\RN)$-boundedness of the variation operators related to $\{e^{-t\Delta^2}\}_{t>0}$ as follows.
\begin{theorem}\label{VB}
Let $\rho>2$ and $0<\lambda<n$.
 If $1<p<\infty$, then $$\|\mathcal{V}_{\rho}(e^{-t\Delta^2})(f)\|_{L^{p,\lambda}(\RN)}\leq C\|f\|_{L^{p,\lambda}(\RN)}.$$
\end{theorem}
\begin{proof}
 For any fixed $x_0\in \RN$ and $r>0$, we write
\begin{equation*}
f(x)=f_0(x)+\sum_{i=1}^\infty f_i(x),
\end{equation*}
where $f_0=f\chi_{B(x_0,2r)}$, $f_i=f\chi_{B(x_0,2^{i+1}r)\setminus B(x_0,2^{i}r)}$ for $i\geq 1$. Then
\begin{align*}
&\quad \quad \Big(\int_{B(x_0,r)}\big|\mathcal{V}_{\rho}(e^{-t\Delta^2})(f)(x)\big|^pdx\Big)^{\frac{1}{p}}\\
&\leq C\Big(\int_{B(x_0,r)}\big|\mathcal{V}_{\rho}(e^{-t\Delta^2})(f_0)(x)\big|^pdx\Big)^{\frac{1}{p}}+C\sum_{i=1}^\infty\Big(\int_{B(x_0,r)}\big|\mathcal{V}_{\rho}(e^{-t\Delta^2})(f_i)(x)\big|^pdx\Big)^{\frac{1}{p}}\\
&=:I+II.
\end{align*}
For $I$, by Theorem \ref{E}, we have
\begin{align*}
I=\int_{B(x_0,r)}\big|\mathcal{V}_{\rho}(e^{-t\Delta^2})(f_0)(x)\big|^pdx
\leq C\int_{B(x_0,2r)}|f(x)|^pdx\leq Cr^{\lambda}\|f\|^p_{L^{p,\lambda}(\RN)}.
\end{align*}
For $II$, we first analyze $\mathcal{V}_{\rho}(e^{-t\Delta^2})(f_i)(x)$. For every $i\geq 1$,
\begin{align*}
\mathcal{V}_{\rho}(e^{-t\Delta^2})(f_i)(x)&=\Big(\sum_{j=0}^\infty\Big|\int_{\RN}\big(b(x-y,t_j)-b(x-y,t_{j+1})\big)f_i(y)dy\Big|^{\rho}\Big)^{\frac{1}{\rho}\nonumber}\\
&\leq C\sum_{j=0}^\infty\int_{\RN}|f_i(y)|\int_{t_{j+1}}^{t_j}\Big|\frac{\partial }{\partial t}b(x-y,t)\Big|dtdy\nonumber\\
&\leq C\int_{B(x_0,2^{i+1}r)\setminus B(x_0,2^{i}r)}|f_i(y)|\int_{0}^{\infty}\Big|\frac{\partial }{\partial t}b(x-y,t)\Big|dtdy.
\end{align*}
Note that for $x\in B(x_0,r)$ and $y\in \RN\setminus B(x_0,2r)$, we know $|x-y|>\frac{1}{2}|x_0-y|$. By using (\ref{kt}), we have
\begin{align*}
\int_{0}^{\infty}\Big|\frac{\partial }{\partial t}b(x-y,t)\Big|dt&=\int_{0}^{|x_0-y|^4}\Big|\frac{\partial }{\partial t}b(x-y,t)\Big|dt+\int_{|x_0-y|^4}^{\infty}\Big|\frac{\partial }{\partial t}b(x-y,t)\Big|dt \nonumber\\
&\leq C\int_{0}^{|x_0-y|^4}t^{-\frac{n}{4}}e^{-A_1 (|x_0-y|t^{-\frac{1}{4}})^\frac{4}{3}}dt+C\int_{|x_0-y|^4}^{\infty}t^{-\frac{n}{4}-1}dt \nonumber\\
&\leq C|x_0-y|^{-n}\int_1^\infty u^{\frac{3n}{4}-1}e^{-A_1 u}du+C|x_0-y|^{-n}\nonumber\\
&\leq C|x_0-y|^{-n}.% \label{es}
\end{align*}
Thus,
\begin{align*}
\mathcal{V}_{\rho}(e^{-t\Delta^2})(f_i)(x)&\leq C\int_{B(x_0,2^{i+1}r)\setminus B(x_0,2^{i}r)}|f_i(y)||x_0-y|^{-n}dtdy\\
&\leq C\left(\int_{B(x_0,2^{i+1}r)}|f_i(y)|^p dy\right)^{\frac{1}{p}}\left(\int_{B(x_0,2^{i+1}r)\setminus B(x_0,2^{i}r)}\frac{1}{|x_0-y|^{np'}}dy^{\frac{1}{p'}}\right)\\
&\leq C(2^i r)^{-\frac{n}{p}}\left(\int_{B(x_0,2^{i+1}r)}|f_i(y)|^p dy\right)^{\frac{1}{p}}.
\end{align*}
Therefore, we have
\begin{align*}
II&\leq C\sum_{i=1}^\infty \left(2^{-in}\int_{B(x_0,2^{i+1}r)}|f_i(y)|^p dy\right)^{\frac{1}{p}}\leq C\sum_{i=1}^\infty \Big(2^{-in}r^\lambda \|f\|^p_{L^{p,\lambda}(\RN)}\Big)^{\frac{1}{p}}\leq Cr^{\frac{\lambda }{p}} \|f\|_{L^{p,\lambda}(\RN)}.
\end{align*}
Consequently,
\begin{align*}
\|\mathcal{V}_{\rho}(e^{-t\Delta^2})(f)\|_{L^{p,\lambda}(\RN)}\leq C\|f\|_{L^{p,\lambda}(\RN)}.
\end{align*}
The proof of this theorem is complete.
\end{proof}

In what follows, we devote to the proof of Theorem \ref{thonm}.

\medskip
\begin{proof}[Proof of Theorem \ref{thonm}]
Without loss of generality, we may assume that $\alpha<0$. Fixing any $x_0\in \RN$ and $r>0$, we write
\begin{equation*}
f(x)=f_0(x)+\sum_{i=1}^\infty f_i(x),
\end{equation*}
where $f_0=f\chi_{B(x_0,2r)}$, $f_i=f\chi_{B(x_0,2^{i+1}r)\setminus B(x_0,2^{i}r)}$ for $i\geq 1$. Then
\begin{align*}
&\Big(\int_{B(x_0,r)}\big|\mathcal{V}_{\rho}(e^{-t\mathcal{L}})(f)(x)\big|^pdx\Big)^{\frac{1}{p}}\\
&\leq C\Big(\int_{B(x_0,r)}\big|\mathcal{V}_{\rho}(e^{-t\mathcal{L}})(f_0)(x)\big|^pdx\Big)^{\frac{1}{p}}
+C\sum_{i=1}^\infty\Big(\int_{B(x_0,r)}\big|\mathcal{V}_{\rho}(e^{-t\mathcal{L}})(f_i)(x)\big|^pdx\Big)^{\frac{1}{p}}\\
&=:I+II.
\end{align*}
From $(i)$ of Theorem \ref{thonp}, we have
\begin{align*}
I
\leq C\int_{B(x_0,2r)}|f(x)|^pdx\leq Cr^{\lambda}\Big(1+\frac{r}{\gamma(x_0)}\Big)^{-\alpha}\|f\|^p_{L^{p,\lambda}_{\alpha,V}(\RN)}.
\end{align*}
For $II$, we first analyze $\mathcal{V}_{\rho}(e^{-t\mathcal{L}})(f_i)(x)$. For every $i\geq 1$,
\begin{align*}%\label{ii}
\mathcal{V}_{\rho}(e^{-t\mathcal{L}})(f_i)(x)&=\Big(\sum_{j=0}^\infty\Big|\int_{\RN}\big(\mathcal{B}_{t_j}(x,y)-
\mathcal{B}_{t_{j+1}}(x,y)\big)f_i(y)dy\Big|^{\rho}\Big)^{\frac{1}{\rho}}\nonumber\\
&\leq C\sum_{j=0}^\infty\int_{\RN}|f_i(y)|\int_{t_{j+1}}^{t_j}\Big|\frac{\partial }{\partial t}\mathcal{B}_t(x,y)\Big|dtdy\nonumber\\
&\leq C\int_{B(x_0,2^{i+1}r)\setminus B(x_0,2^{i}r)}|f_i(y)|\int_{0}^{\infty}\Big|\frac{\partial }{\partial t}\mathcal{B}_t(x,y)\Big|dtdy.
\end{align*}
Note that for $x\in B(x_0,r)$ and $y\in \RN\setminus B(x_0,2r)$, we have $\displaystyle |x-y|>\frac{1}{2}|x_0-y|$. We discuss $\displaystyle \int_{0}^{\infty}|{\partial_ t}\mathcal{B}_t(x,y)|dt$ in two cases.
For the one case: $|x_0-y|\leq \gamma(x_0)$, by $(ii)$ of Lemma \ref{k6}, we have
\begin{align}
\int_{0}^{\infty}\Big|\frac{\partial}{\partial t}\mathcal{B}_t(x,y)\Big|dt&=\int_{0}^{|x_0-y|^4}\Big|\frac{\partial}{\partial t}\mathcal{B}_t(x,y)\Big|dt+\int_{|x_0-y|^4}^{\infty}\Big|\frac{\partial}{\partial t}\mathcal{B}_t(x,y)\Big|dt \nonumber\\
&\leq C\int_{0}^{|x_0-y|^4}t^{-\frac{n}{4}-1}e^{-A_1 (|x_0-y|t^{-\frac{1}{4}})^\frac{4}{3}}dt+C\int_{|x_0-y|^4}^{\infty}t^{-\frac{n}{4}-1}dt \nonumber\\
&\leq C|x_0-y|^{-n}+C\int_{0}^{|x_0-y|^4}t^{-\frac{n}{4}-1}\Big(\frac{t^{1/3}}{|x_0-y|^{4/3}}\Big)^{{3(n+4)}/{4}}dt\nonumber\\
&\leq C|x_0-y|^{-n}\Big(1+\frac{|x_0-y|}{\gamma(x_0)}\Big)^{-N}. \label{esforB}
\end{align}
For the other case: $|x_0-y|\geq \gamma(x_0)$, applying $(ii)$ of Lemma \ref{k6} together with Lemma \ref{veq}, we have
\begin{align}
\int^\infty_{|x_0-y|^4}\Big|\frac{\partial}{\partial t}\mathcal{B}_t(x,y)\Big|dt
&\leq C\int^\infty_{|x_0-y|^4}t^{-\frac{n}{4}-1}\Big(1+\frac{\sqrt{t}}{\gamma^2(y)}\Big)^{-k}e^{-A_1 \big(|x_0-y|t^{-\frac{1}{4}}\big)^{4/3}}dt \nonumber\\
&\leq C\Big(1+\frac{{|x_0-y|^2}}{\gamma^2(y)}\Big)^{-k}|x_0-y|^n\nonumber\\
&\leq C\Big(1+\frac{\big(\frac{|x_0-y|^2}{\gamma{(x_0)}}\big)^2}{c_0(1+\frac{|x_0-y|^2}{\gamma{(x_0)}})^{\frac{2k_0}{k_0+1}}}\Big)^{-k}|x_0-y|^n\nonumber\\
&\leq  C|x_0-y|^n\Big(1+\frac{|x_0-y|^2}{\gamma{(x_0)}}\Big)^{-N},\label{eqB1}
\end{align}
where we take $\displaystyle N=\left[\frac{k(k_0-1)}{k_0+1}\right]$ for any $k\in \mathbb{{N}}$.
And
\begin{align}
&\int_{0}^{|x_0-y|^4}\Big|\frac{\partial}{\partial t}\mathcal{B}_t(x,y)\Big|dt\nonumber\\
&=\int_{0}^{\gamma^4(x_o)}\Big|\frac{\partial}{\partial t}\mathcal{B}_t(x,y)\Big|dt+\int_{\gamma^4(x_0)}^{|x_0-y|^4}\Big|\frac{\partial}{\partial t}\mathcal{B}_t(x,y)\Big|dt\nonumber\\
&\leq C\int_{0}^{\gamma^4(x_0)}t^{-\frac{n}{4}-1}e^{-A_1 \big(|x_0-y|t^{-\frac{1}{4}}\big)^{4/3}}dt+C\int_{\gamma^4(x_0)}^{|x_0-y|^4}t^{-\frac{n}{4}-1}e^{-A_1 \big(|x_0-y|t^{-\frac{1}{4}}\big)^{4/3}}dt \nonumber\\
&\leq C\int^\infty_{\frac{|x_0-y|^{4/3}}{\gamma(x_0)^{4/3}}}{|x_0-y|^{-n}}u^{\frac{3n}{4}-1}e^{-A_1u}du+C\gamma(x_0)^{-n-4}
e^{-A_1\frac{|x_0-y|^{4/3}}{\gamma(x_0)^{4/3}}}|x_0-y|^4\nonumber\\
&\leq  C |x_0-y|^{-n}e^{-c\frac{|x_0-y|^{4/3}}{\gamma(x_0)^{4/3}}}+C\gamma(x_0)^{-n-4}
e^{-A_1\frac{|x_0-y|^{4/3}}{\gamma(x_0)^{4/3}}}|x_0-y|^4\nonumber\\
&\leq  C |x_0-y|^{-n}\left(1+\frac{|x_0-y|}{\gamma(x_0)}\right)^{-N}.\label{eqB2}
\end{align}
Combining (\ref{esforB}), (\ref{eqB1}) and (\ref{eqB2}), we have
\begin{align*}
&\int_{B(x_0,2^{i+1}r)\setminus B(x_0,2^{i}r)}|f_i(y)|\int_{0}^{\infty}\Big|\frac{\partial }{\partial t}\mathcal{B}_t(x,y)\Big|dy\\
&\leq C\int_{B(x_0,2^{i+1}r)\setminus B(x_0,2^{i}r)}|x_0-y|^{-n}\Big(1+\frac{|x_0-y|}{\gamma(x_0)}\Big)^{-N}|f_i(y)|dy\\
&\leq C(2^ir)^{-\frac{n}{p}}\Big(1+\frac{2^ir}{\gamma(x_0)}\Big)^{-N}\Big(\int_{B(x_0,2^{i+1}r)}|f_i(y)|^pdy\Big)^{\frac{1}{p}}.
\end{align*}
Thus, taking $N=[-\alpha]+1$, we get
\begin{align*}
\int_{B(x_0,r)}\big|\mathcal{V}_{\rho}(e^{-t\mathcal{L}})(f_i)(x)\big|^pdx
&\leq C2^{-ni}\Big(1+\frac{2^ir}{\gamma(x_0)}\Big)^{-Np}\int_{B(x_0,2^{i+1}r)}|f_i(y)|^pdy\\
&\leq C2^{(\lambda-n)i}r^\lambda\Big(1+\frac{2^ir}{\gamma(x_0)}\Big)^{-Np-\alpha}\|f\|^p_{L^{p,\lambda}_{\alpha,V}(\RN)}\\
&\leq C2^{(\lambda-n)i}r^\lambda\Big(1+\frac{r}{\gamma(x_0)}\Big)^{-\alpha}\|f\|^p_{L^{p,\lambda}_{\alpha,V}(\RN)}.
\end{align*}
Since $\lambda<n$, we have $II\leq C\|f\|_{L^{p,\lambda}_{\alpha,V}(\RN)}$. Hence,
\begin{align*}
\|\mathcal{V}_\rho(e^{-t\mathcal{L}})(f)\|_{L^{p,\lambda}_{\alpha,V}(\RN)}\leq C\|f\|_{L^{p,\lambda}_{\alpha,V}(\RN)}.
\end{align*}
The proof of the theorem is completed.
\end{proof}

Finally, we can give  the boundedness of the variation operators related to generalized Poisson operators  $\mathcal{P}_{t,\mathcal{L}}^\sigma$ in the Morrey spaces as follows.
\begin{theorem}
Let $V\in RH_{q_0}(\RN)$ for $q_0\in (n/2, \infty)$, $n\geq 5$ and $\rho>2$. Assume that $\alpha\in \mathbb{\mathbb{R}}$ and $\lambda\in (0,n)$.
There exists a constant $C>0$ such that
$$\|\mathcal{V}_{\rho}(\mathcal{P}_{t,\mathcal{L}}^\sigma)(f)\|_{L^{p,\lambda}_{\alpha, V}(\RN)}\leq C\|f\|_{L^{p,\lambda}_{\alpha, V}(\RN)}, \quad 1< p<\infty.$$
\end{theorem}
\begin{proof}
We can prove this theorem as the same procedure in the proof of Theorem \ref{general}.
\end{proof}

\end{document}